\documentclass[11pt]{article}

\usepackage{amsmath,amssymb,amsthm,mathtools}
\usepackage[margin=1in]{geometry}
\usepackage{hyperref}
\usepackage{url}
\usepackage{authblk}

\newcommand{\R}{\mathbb{R}}
\newcommand{\N}{\mathbb{N}}
\newcommand{\Pbb}{\mathbb{P}}
\newcommand{\Ebb}{\mathbb{E}}
\newcommand{\cN}{\mathcal{N}}

\newtheorem{theorem}{Theorem}
\newtheorem{lemma}[theorem]{Lemma}
\newtheorem{proposition}[theorem]{Proposition}
\newtheorem{corollary}[theorem]{Corollary}
\newtheorem{remark}[theorem]{Remark}

\newcommand{\phin}{\varphi}
\newcommand{\Phib}{\overline{\Phi}}

\title{The sharp one-dimensional convex sub-Gaussian comparison constant}
\author[1]{Damek Davis\thanks{\texttt{damek@wharton.upenn.edu}. Research supported by NSF DMS award 2523384.}}
\author[2]{Sam Power}
\affil[1]{Department of Statistics and Data Science, The Wharton School, University of Pennsylvania}
\affil[2]{School of Mathematics, University of Bristol}
\date{}

\begin{document}
\maketitle

\begin{abstract}
Let $X$ be an integrable real random variable with mean zero and two-sided sub-Gaussian tail
$\mathbb{P}(|X|>t)\le 2e^{-t^{2}/2}$ for all $t\ge 0$.
We determine the smallest constant $c_\star$ such that $X$ is dominated in convex order by $c_\star G$,
where $G$ is standard normal.
Equivalently, $c_\star^2$ is the sharp one-dimensional convex sub-Gaussian comparison constant
appearing in the \emph{Optimization Constants in Mathematics} repository~\cite{optimization-constants-repo}.
We show that $c_\star$ is given by an explicit system of one-dimensional equations and is attained by an extremal distribution that saturates the tail constraint.
Numerically, $c_\star \approx 2.30952$ (so $c_\star^2 \approx 5.33386$). We also determine the analogous sharp constant under a two-sided sub-exponential tail bound, with convex domination by a scaled Laplace law.
Finally, we record two higher-dimensional consequences: a sequential tensorization principle for multivariate convex domination, and a dimension-free Gaussian comparator for the cone generated by convex ridge functions (the linear convex order).
\end{abstract}

\section{Introduction}

A recent theorem of van Handel \cite{vanHandelSubgComparison} shows that if $X$ is a random vector in $\R^d$ such that
$\Pbb (|\langle v, X \rangle| > t) \le 2e^{-t^{2}/2}$ for all $\|v\|_2 = 1$ and $t\ge 0$, then $X$ is dominated in convex order by a universal constant times a standard Gaussian vector.
The optimal value of this universal constant is not known, even in dimension~$1$.

This note resolves the one-dimensional case.
We compute the sharp constant and exhibit an extremal distribution.
The argument is elementary and rests on two classical facts: (i) one-dimensional convex order is equivalent to comparison of the stop-loss transforms $u\mapsto \Ebb[(X-u)_+]$ \cite[Ch.~3]{ShakedShanthikumar}; (ii) under a two-sided tail constraint, the stop-loss transform is maximized by a distribution that saturates the constraint and has a single flat region.

\subsection*{Setup and the constant}

Throughout, $G \sim \cN (0,1)$, $\phin (x) := (2\pi)^{-1/2} e^{-x^{2}/2}$ is the standard normal density, and $\Phib(x) := \int_x^\infty \phin(t) \, \mathrm{d}t$ is the Gaussian tail.
We call $X$ \emph{$1$-sub-Gaussian in the tail sense} if
\begin{equation}
\label{eq:subg_tail}
\Ebb[X] = 0
\qquad\text{and}\qquad
\Pbb(|X|>t) \le s_G(t) := \min \{ 1,  2e^{-t^2/2} \} \qquad\text{for all } t \ge 0.
\end{equation}
Define the one-dimensional comparison constant
\begin{equation}
\label{eq:def_cstar}
c_\star := \inf \Bigl\{ c > 0 :\ \text{every $X$ satisfying \eqref{eq:subg_tail} obeys } X \preceq_{cx} cG \Bigr\},
\end{equation}
where $X\preceq_{cx}Y$ denotes convex domination: $\Ebb[f(X)] \le \Ebb[f(Y)]$ for every convex $f:\R\to\R$ for which both expectations are finite.

Let $t_0:=\sqrt{2\log 2}$ denote the point at which $s_G$ first descends from $1$, and hence define
\begin{equation}
\label{eq:Adef_Bdef}
A := \int_0^\infty s_G (t) \, \mathrm{d}t
= t_0 + 2 \int_{t_0}^\infty e^{-t^2/2} \, \mathrm{d}t,
\qquad
B := \frac{A}{2}.
\end{equation}
For $x\ge t_0$, set
\begin{equation}
\label{eq:Hdef}
H(x) := 2x \, e^{-x^2/2} + 2 \int_x^\infty e^{-t^2/2} \, \mathrm{d}t.
\end{equation}
Since $H$ is strictly decreasing with $H(t_0) = A$ and $\lim_{x\to\infty} H(x) = 0$, the intermediate value theorem supplies a unique $a > t_0$ such that $H(a) = B$. Define then
\begin{equation}
\label{eq:p0def}
p_0 := 2 e^{-a^2/2} \in (0,1),
\end{equation}
let $z$ be the unique solution of $\Phib(z) = p_0$, and set
\begin{equation}
\label{eq:c0def}
c_0 := \frac{B}{\phin(z)}.
\end{equation}

We now state the main result of this note.
\begin{theorem}[Sharp one-dimensional convex sub-Gaussian comparison]
\label{thm:main}
The sharp constant in \eqref{eq:def_cstar} satisfies $c_\star = c_0$, where $c_0$ is defined by \eqref{eq:Adef_Bdef}--\eqref{eq:c0def}. In particular:
\begin{enumerate}
\item For every random variable $X$ satisfying \eqref{eq:subg_tail} and every convex $f:\R\to\R$,
\begin{equation}
\label{eq:main_conv_dom}
\Ebb[f(X)] \le \Ebb[f(c_0G)]
\end{equation}
whenever the right-hand side is finite.
\item For every $c<c_0$, there exist a random variable $X^\star$ satisfying \eqref{eq:subg_tail} and a convex
function $f$ such that $\Ebb[f(X^\star)] > \Ebb[f(cG)]$.
One may take $f(x) = (x-cz)_+$, where $z$ is the parameter from \eqref{eq:c0def}.
\end{enumerate}
Consequently, the one-dimensional value of the constant $C_{48}$ in \cite{optimization-constants-repo} is $C_{48}^{(1)} = c_0^2$.
\end{theorem}

\begin{remark}[Numerical value]
A direct high-precision evaluation of \eqref{eq:Adef_Bdef}--\eqref{eq:c0def} gives
\[
a \approx 1.80334,\qquad
p_0 \approx 0.39342,\qquad
z \approx 0.27041,\qquad
c_0 \approx 2.30952,\qquad
c_0^2 \approx 5.33386.
\]
No numerical computation is used in the derivation of the exact characterization $c_\star=c_0$.
\end{remark}

\begin{remark}[Other notions of sub-Gaussianity]
    Note that a sub-Gaussian bound on the moment-generating function of $X$ of the form
    \begin{equation*}
        \mathbb{E} [ e^{\lambda X}] \leq e^{\lambda^2 / 2} \qquad \text{for all } \lambda \in \mathbb{R}
    \end{equation*}
    immediately implies the tail constraint with which we work. Whether sharper convex orderings can be found under this nominally stronger assumption is an interesting question, and seems likely to admit different extrema.
\end{remark}

Section~\ref{sec:psi1} records the analogous sharp constant under a two-sided sub-exponential tail bound, with convex domination by a scaled Laplace law. The case of general $d \geq 2$ remains open.

\section{Convex order and stop-loss transforms}

We recall the characterization of one-dimensional convex order in terms of stop-loss transforms. For $u\in\R$ define the hinge function $(x-u)_+ := \max \{ x - u, 0\}$.

\begin{proposition}[Stop-loss characterization of convex domination]
\label{prop:stoploss_to_convex}
Let $X,Y$ be integrable real random variables. Then $X \preceq_{cx} Y$ if and only if
\begin{equation}
\label{eq:stoploss_iff}
\Ebb[X] = \Ebb[Y]
\qquad\text{and}\qquad
\Ebb[(X-u)_+] \le \Ebb[(Y-u)_+]\quad\text{for all }u\in\R.
\end{equation}
\end{proposition}

\begin{proof}
This is standard; see, e.g., \cite[Thm.~3.A.1]{ShakedShanthikumar}. For completeness, we recall the short direction needed below.
Assume \eqref{eq:stoploss_iff}.  Any convex $f : \R \to \R$ admits the representation
\begin{equation}
\label{eq:hinge_representation}
f(x) = \alpha x + \beta + \int_{\R} (x-u)_+ \, \mu(\mathrm{d}u),
\end{equation}
where $\alpha, \beta \in \R$ and $\mu$ is a nonnegative Borel measure on $\R$
\cite[Prop.~3.A.4]{ShakedShanthikumar}.
Integrability of $X,Y$ ensures that $\Ebb [ \alpha X + \beta ] = \Ebb [ \alpha  Y + \beta ]$. Tonelli's Theorem and \eqref{eq:stoploss_iff} yield that
\[
\Ebb[f(X)]
= \alpha \Ebb[X] + \beta + \int_{\R} \Ebb [ (X - u)_+ ] \, \mu (\mathrm{d}u)
\le \alpha\Ebb[Y]+\beta+\int_{\R}\Ebb[(Y-u)_+]\,\mu( \mathrm{d}u)
= \Ebb[f(Y)],
\]
whenever $\Ebb [ f(Y) ]<\infty$.  This is the desired convex domination.
\end{proof}

We next collect three elementary lemmas that will be used repeatedly in subsequent sections.

\begin{lemma}[Layer-cake for hinges]
\label{lem:layercake}
Let $X$ be integrable and let $u \in \R$. Then
\[
\Ebb [ (X - u)_+] = \int_u^\infty \Pbb (X > t) \, \mathrm{d}t.
\]
\end{lemma}

\begin{proof}
This is the layer-cake identity applied to the nonnegative random variable $(X - u)_+$, namely
\[
\Ebb[(X-u)_+] = \int_0^\infty \Pbb((X - u)_+ > s) \, \mathrm{d}s = \int_0^\infty \Pbb(X > u + s)\, \mathrm{d}s = \int_u^\infty \Pbb(X > t)\, \mathrm{d}t.
\]
\end{proof}

\begin{lemma}[Tangent line lower bound]\label{lem:tangent}
Let $I \subseteq \R$ be an interval and let $g : I \to \R$ be convex and differentiable.
Fix $B \in \R$ and $p_0 \in \R$. If there exists $u_\star \in I$ such that
$g(u_\star) = B - p_0 u_\star$ and $g^{\prime} (u_\star) = -p_0$, then
\[
g(u) \ge B - p_0 u \qquad \text{for all } u \in I.
\]
\end{lemma}

\begin{proof}
By convexity, for all $u \in I$,
\[
g(u) \ge g(u_\star) + g^{\prime} (u_\star) (u - u_\star) = B - p_0u.
\]
\end{proof}

The following monotone-ratio principle will allow us to deduce $D\ge 0$ from the sign pattern of $D^{\prime}$.

\begin{lemma}[Monotone-ratio principle]
\label{lem:monotone_ratio_principle}
Let $D : [a, \infty) \to \R$ be differentiable with $\lim_{u\to\infty} D(u) = 0$.
Assume $D^{\prime}(u) = w(u) \bigl(1 - R(u) \bigr)$ for $u \ge a$, where $w(u) > 0$ and $R$ is nondecreasing.
If $D(a) \ge 0$, then $D(u) \ge 0$ for all $u \ge a$.
\end{lemma}

\begin{proof}
If $D$ attains a negative value, let $u_0$ be a point where $D$ achieves its minimum on $[a,\infty)$.
Then $D(u_0) < 0$ and necessarily $u_0 > a$ and $D^{\prime}(u_0) = 0$.
Since $w(u_0) > 0$, we have $R(u_0) = 1$.
By monotonicity of $R$, $R(u) \le 1$ for $u \le u_0$ and $R(u) \ge 1$ for $u\ge u_0$. 
As such, $D^{\prime}(u) \ge 0$ for $u\le u_0$ and $D^{\prime}(u) \le 0$ for $u\ge u_0$, so that $u_0$ is instead a global maximum of $D$ on $[a, \infty)$, contradicting the assumption that $D(u_0) < 0$ and $\lim_{u\to\infty}D(u)=0$.
\end{proof}

\section{A sharp stop-loss envelope under the two-sided tail constraint}

Fix a deterministic function \(s : [0, \infty) \to [0, 1]\). We interpret \(s\) as a two-sided tail envelope: an integrable random variable \(X\) `satisfies the tail constraint \(s\)' if
\[
\Pbb (|X|>t) \le s(t) \qquad \text{for all } t \ge 0.
\]
Lemma~\ref{lem:envelope} gives a sharp upper envelope \(J_s\) for the stop-loss transform \(u \mapsto \Ebb[(X-u)_+]\) over all mean-zero \(X\) obeying this constraint. The sub-Gaussian and sub-exponential comparisons later are obtained by specializing \(s\) to \(s_G : t \mapsto \min\{1, 2e^{-t^2/2}\}\) and \(s_E : t\mapsto \min \{1, 2e^{-t} \} \), respectively.

\begin{lemma}[Sharp stop-loss envelope]\label{lem:envelope}
Let $s:[0,\infty)\to[0,1]$ be non-increasing and continuous, and assume
\[
A_s: = \int_0^\infty s(t)\, \mathrm{d}t < \infty,
\qquad
B_s := \frac{A_s}{2}.
\]
Define
\[
H_s(x) := x \, s(x) + \int_x^\infty s(t) \, \mathrm{d}t.
\]
Assume there exists $a>0$ such that $H_s(a) = B_s$, and set $p_0: = s(a)$. Define $J_s: [0, \infty) \to \R_+$ by
\[
J_s(u):=
\begin{cases}
B_s - p_0 u, & 0 \le u \le a,\\[2mm]
\int_u^\infty s(t) \, \mathrm{d}t, & u \ge a.
\end{cases}
\]
If $X$ is integrable with $\Ebb[X] = 0$ and satisfies the tail constraint
\[
\Pbb \bigl( |X| > t \bigr) \le s(t) \qquad \text{for all } t \ge 0,
\]
then for every $u \ge 0$, there holds the stop-loss bound
\[
\Ebb[ (X - u)_+ ] \le J_s (u).
\]
\end{lemma}

\begin{proof}
Let $p(t) := \Pbb(X > t)$. Then, by assumption, $p$ is nonincreasing and $p(t) \le s(t)$ for all $t \ge 0$. By Lemma~\ref{lem:layercake}, we compute
\[
\int_0^\infty p(t)\, \mathrm{d}t=\Ebb[X_+].
\]
Since $\Ebb[X]=0$, it holds that $\Ebb|X| = 2\,\Ebb[X_+]$. Another application of Lemma~\ref{lem:layercake} yields the bound
\[
\Ebb|X| = \int_0^\infty \Pbb(|X| > t)\, \mathrm{d}t \le \int_0^\infty s(t) \, \mathrm{d}t = A_s,
\]
and hence $\int_0^\infty p(t) \, \mathrm{d}t\le A_s/2 = B_s$.
Fix now $u\ge 0$ and set $\alpha := p(u)$. Since $p$ is nonincreasing, we can make the elementary bound
\[
\int_0^u p(t)\, \mathrm{d}t\ge \alpha u,
\qquad\text{hence}\qquad
\int_u^\infty p(t)\, \mathrm{d}t\le B_s-\alpha u.
\tag{M}\label{eq:mass_bound}
\]
Separately, since $p(t)\le \alpha$ for $t\ge u$, we can equally bound
\[
\int_u^\infty p(t) \, \mathrm{d}t \le \int_u^\infty \min \{\alpha, s(t)\}\, \mathrm{d}t.
\tag{C}\label{eq:cap_bound}
\]
If $u \ge a$, then \eqref{eq:cap_bound} gives immediately that $\int_u^\infty p(t) \, \mathrm{d}t \le \int_u^\infty s(t)\, \mathrm{d}t = J_s(u)$.

Assume then that $u<a$. If $\alpha\ge p_0$, then \eqref{eq:mass_bound} yields that
\[
\int_u^\infty p(t) \, \mathrm{d}t \le B_s - \alpha u \le B_s-p_0 u = J_s (u).
\]

Finally, assume that $u<a$ and $\alpha < p_0$.
If $\alpha = 0$, then $p(t) = 0$ for all $t \ge u$, so $\int_u^\infty p(t)\,\mathrm{d}t = 0 \le J_s(u)$ and there is nothing to prove.
Assume therefore that $\alpha > 0$. By continuity and monotonicity of $s$ and $\lim_{t\to\infty} s(t) = 0$, there exists $b > a$ such that $s(b) = \alpha$.
\eqref{eq:cap_bound} then gives that
\[
\int_u^\infty p(t)\, \mathrm{d}t\le \alpha(b-u)+\int_b^\infty s(t)\, \mathrm{d}t = H_s(b)-\alpha u.
\]
Since $s$ is nonincreasing and $s(b) = \alpha$, we have the elementary bound $\int_a^b s(t)\, \mathrm{d}t\ge \alpha(b-a)$, and can hence deduce that
\begin{align*}
B_s - H_s(b)
&=H_s(a) - H_s(b)
= a p_0 - b \alpha + \int_a^b s(t) \, \mathrm{d}t\\
&\ge a p_0 - b \alpha + \alpha (b-a)
= a(p_0 - \alpha)
\ge u (p_0 - \alpha),
\end{align*}
using $u<a$ in the last step. Rearrangement then gives that $H_s(b) - \alpha u \le B_s - p_0 u = J_s (u)$, and we conclude.
\end{proof}

\begin{lemma}[A global linear lower bound for $J_s$]\label{lem:J_ge_line}
Work in the setting of Lemma~\ref{lem:envelope}. Then
\[
J_s (u) \ge B_s - p_0 u \qquad \text{for all } u \ge 0.
\]
\end{lemma}

\begin{proof}
By definition, $J_s (u) = B_s - p_0 u$ for $u \in [0, a]$. For $u \ge a$, we can write $J_s (u) = \int_u^\infty s(t)\, \mathrm{d}t$ and $J_s (a) = B_s - p_0 a$. Since $s$ is nonincreasing, we can decompose
\[
J_s (u) = J_s (a) - \int_a^u s(t) \, \mathrm{d}t \ge B_s -p_0 a - p_0 (u - a) = B_s - p_0 u.
\]
\end{proof}

\begin{lemma}[Extremizer attaining $J_s$]\label{lem:extremizer}
Work in the setting of Lemma~\ref{lem:envelope}, and assume in addition that there exists $t_0 \ge 0$ such that $s(t) = 1$ for $t \in [0, t_0]$ and $s$ is strictly decreasing on $[t_0, \infty)$. 
Assume also that the solution $a$ to $H_s (a) = B_s$ satisfies $a > t_0$.
Let $X^\star$ be the random variable with distribution function
\[
F_\star(x):=
\begin{cases}
0,& x \le -a,\\[1mm]
s (-x) - p_0, & -a < x \le -t_0,\\[1mm]
1 - p_0, & -t_0 < x < a,\\[1mm]
1 - s (x), & x \ge a.
\end{cases}
\]
The function $F_\star$ is a cumulative distribution function on $\R$.
Then $\Pbb( |X^\star| > t) = s (t)$ for all $t \ge 0$ and $\Ebb [X^\star] = 0$. Moreover, for every $u\ge 0$, the stop-loss satisfies
\[
\Ebb [ (X^\star - u)_+] = J_s (u).
\]
\end{lemma}

\begin{proof}
We first verify the two-sided tail by considering cases. If $0 \le t< t_0$, then $X^\star \notin(-t_0, a)$ almost surely, so
$\Pbb(|X^\star| > t) = 1 = s (t)$. If $t_0 \le t <a$, then $\Pbb(X^\star > t) = p_0$ and
\[
\Pbb(X^\star<-t) = F_\star (-t) = s (t) - p_0,
\]
and so $\Pbb( |X^\star| > t) = s (t)$. If $t \ge a$, then $\Pbb(X^\star > t) = s(t)$ and $\Pbb(X^\star < -t) = 0$, so again $\Pbb( |X^\star| > t) = s(t)$.

By Lemma~\ref{lem:layercake} and the identity $\Pbb( |X^\star| > t) = s (t)$, compute that
\[
\Ebb |X^\star| = \int_0^\infty \Pbb( |X^\star| > t) \, \mathrm{d}t = \int_0^\infty s(t) \, \mathrm{d}t = A_s,
\]
and also that
\[
\Ebb [X^\star_+] = \int_0^\infty \Pbb(X^\star > t) \, \mathrm{d}t
= \int_0^a p_0 \, \mathrm{d}t + \int_a^\infty s(t) \, \mathrm{d}t
= a p_0 + \int_a^\infty s(t) \, \mathrm{d}t
= H_s(a) = B_s.
\]
We thus see that (writing $X^\star_- = \max (-X^\star, 0)$) $\Ebb[X^\star_-]=\Ebb|X^\star|-\Ebb[X^\star_+] = A_s - B_s = B_s$, and hence that $\Ebb [X^\star] = 0$.

Finally, for $u \ge 0$, Lemma~\ref{lem:layercake} gives that
\[
\Ebb [(X^\star - u)_+] = \int_u^\infty \Pbb(X^\star > t) \, \mathrm{d}t.
\]
If $u < a$, then this equals $p_0 (a - u) + \int_a^\infty s(t) \, \mathrm{d}t = B_s - p_0 u = J_s (u)$, whereas if $u\ge a$, then it equals $\int_u^\infty s(t) \, \mathrm{d}t = J_s (u)$, i.e. equality holds throughout.
\end{proof}

\begin{proposition}[From envelope domination to convex domination]\label{prop:envelope_to_cx}
Work in the setting of Lemma~\ref{lem:envelope} and write $J_s$ for the corresponding envelope.
Let $X$ be integrable with $\Ebb [X] = 0$ and satisfy the tail constraint $\Pbb( |X| > t) \le s(t)$ for all $t\ge 0$.
Let $Y$ be integrable and symmetric about $0$, and set $g_Y (u) := \Ebb[(Y - u)_+]$.
If
\[
g_Y(u) \ge J_s(u) \qquad\text{for all } u\ge 0,
\]
then $X \preceq_{cx} Y$.
\end{proposition}

\begin{proof}
For $u \ge 0$, Lemma~\ref{lem:envelope} gives $\Ebb [(X - u)_+] \le J_s(u)\le g_Y(u) = \Ebb[(Y - u)_+]$.
Applying Lemma~\ref{lem:envelope} to $-X$ (which also has $\Ebb [ -X ] = 0$ and satisfies the same tail constraint as $X$) yields
$\Ebb[(- X - u)_+] \le J_s (u) \le g_Y (u) =\Ebb[(- Y - u)_+ ]$ for all $u \ge 0$, using symmetry of $Y$.
Fix $u \in \R$. If $u \ge 0$ then the preceding display gives $g_X (u) \le g_Y (u)$.
If $u < 0$, then set $v := - u > 0$ and use the identity
\[
(x + v)_+ = (- x - v)_+ + x + v
\]
to write
\[
g_X(-v) = \Ebb [(X + v)_+] = \Ebb [(- X - v)_+] + v = g_{-X}(v) + v,
\qquad
g_Y (-v) = g_{-Y} (v) + v,
\]
since $\Ebb [X] = \Ebb [Y] = 0$. The bound for $-X$ yields $g_{-X} (v) \le g_{-Y} (v)$ for all $v\ge 0$, and hence $g_X(u) \le g_Y(u)$ for all $u \in \R$. Applying Proposition~\ref{prop:stoploss_to_convex} then completes the proof.
\end{proof}

\section{Gaussian domination and proof of Theorem~\ref{thm:main}}

By Proposition~\ref{prop:envelope_to_cx}, Theorem~\ref{thm:main} follows once we show that $g_{c_0}(u) \ge J_G (u)$ for all $u \ge 0$, where $J_G$ is the stop-loss envelope from Lemma~\ref{lem:envelope} for the sub-Gaussian tail envelope $s_G : t\mapsto \min\{1, 2e^{-t^2/2}\}$. This section proves this inequality and identifies the sharp $c_0$.

For $c>0$ define the Gaussian stop-loss transform
\begin{equation}
\label{eq:gcdef}
g_c(u) := \Ebb[(c G - u)_+], \qquad u \in\R.
\end{equation}
We first recall an exact formula for $g_c$.

\begin{lemma}[Gaussian stop-loss formula]
\label{lem:gaussian_call}
For $c > 0$ and $u \in \R$,
\begin{equation}
\label{eq:gaussian_call_formula}
g_c(u) =c \, \phin (u/c) - u \, \Phib(u/c).
\end{equation}
In particular, $g_c$ is convex, differentiable, and satisfies
\begin{equation}
\label{eq:gc_derivative}
g_c^{\prime}(u)=-\Phib(u/c).
\end{equation}
\end{lemma}

\begin{proof}
Let $Z := c G$.  By Lemma~\ref{lem:layercake},
\[
g_c(u) = \int_u^\infty \Pbb(Z>t)\, \mathrm{d}t = \int_u^\infty \Phib(t/c)\, \mathrm{d}t.
\]
Differentiate to obtain $g_c^{\prime} (u)=-\Phib(u/c)$.
Integrating by parts in the last display gives \eqref{eq:gaussian_call_formula}.
\end{proof}

\begin{remark}[Crude bounds]
\label{rem:crude_bounds}
We record a short verification of $a > \sqrt{2}$ and $c_0 > \sqrt{2}$; these numerical bounds will be used in subsequent developments. 
First, because $H(a) = B$ and $H$ is strictly monotone, it suffices to show that $H(\sqrt{2})>B$.
Since $H (\sqrt{2}) > 2\sqrt{2}e^{-1}$ and
\[
B = \frac{t_0}{2} + \int_{t_0}^\infty e^{-t^2/2} \, \mathrm{d}t
\le \frac{t_0}{2} + \frac{e^{-t_0^2/2}}{t_0} = \frac{t_0}{2} + \frac{1}{2t_0},
\qquad t_0 = \sqrt{2\log 2} \in (1.17,1.18),
\]
one checks readily that $B < 1.02 < 2\sqrt{2}e^{-1} < H(\sqrt{2})$, and hence that $a > \sqrt{2}$.
Similarly, setting $x_1 := \sqrt{2\log 4}$ so that $e^{-x_1^2/2}=1/4$, one writes
$H(x_1) = \sqrt{\log 2} + 2\int_{x_1}^\infty e^{-t^2/2}\,\mathrm{d}t$.
The bound $\phin(x)/\Phib(x)\le x+1/x$ from \cite[Eq.~7.1.13]{AbramowitzStegun}
rearranges to $\int_x^\infty e^{-t^2/2}\,\mathrm{d}t\ge \frac{x}{x^2+1}\,e^{-x^2/2}$,
whence
$H(x_1) \ge \sqrt{\log 2}\cdot\frac{4\log 2 + 2}{4\log 2 + 1} > 1.05 > B$,
and so $a > x_1$.
Consequently $p_0 = 2e^{-a^2/2} < 2e^{-\log 4} = 1/2$, confirming $z > 0$.
For the second inequality, use that $\phin(z)\le \phin(0)=(2\pi)^{-1/2}$ and $B \ge t_0/2 > 0.58$ to see that
$c_0 = B / \phin(z) \ge B \sqrt{2\pi} > 1.45 > \sqrt{2}$.
\end{remark}

\begin{lemma}[A monotone ratio]
\label{lem:Rmono}
Define
\[
R(u) := \frac{\Phib(u/c_0)}{2e^{-u^2/2}} \qquad(u \ge a).
\]
Then $R$ is nondecreasing on $[a, \infty)$.
\end{lemma}

\begin{proof}
Differentiate $\log R(u) = \log \Phib (u / c_0) + u^2/2 - \log 2$ to obtain that
\[
\frac{d}{du} \log R(u) = u - \frac{1}{c_0} \frac{\phin(u/c_0)}{\Phib(u/c_0)}.
\]
The Mills ratio bound $\phin(x)/\Phib(x)\le x+1/x$ for $x>0$
\cite[Eq.~7.1.13]{AbramowitzStegun} yields that for $u\ge a$,
\[
\frac{d}{du}\log R(u)
\ge u - \frac{1}{c_0} \left(\frac{u}{c_0} + \frac{c_0}{u}\right)
= u \Bigl(1 - \frac{1}{c_0^2}\Bigr) - \frac{1}{u}.
\]
By Remark~\ref{rem:crude_bounds}, one sees that for $u\ge a$, it holds that $u^2(1-1/c_0^2)>1$, and so this expression is strictly positive.
It thus follows that $R$ is increasing on this same interval.
\end{proof}

\begin{lemma}[Gaussian stop-loss dominates the envelope]\label{lem:gauss_ge_J}
For all $u \ge 0$, we have $g_{c_0} (u) \ge J_G (u)$.
\end{lemma}

\begin{proof}
Set $u_\star := c_0 z$. 
By Lemma~\ref{lem:gaussian_call}, $g_{c_0}$ is convex and differentiable on $\R$ with
$g_{c_0}^{\prime}(u) = - \Phib(u/c_0)$. 
Since $\Phib(z) = p_0$, we check that $g_{c_0}^{\prime}(u_\star) = - p_0$. 
Moreover,
\[
g_{c_0} (u_\star)
=c_0 \, \phin(z) - u_\star \, \Phib(z)
=c_0 \, \phin(z) - p_0 u_\star
=B - p_0 u_\star,
\]
using the definition $c_0 = B/ \phin(z)$. Lemma~\ref{lem:tangent} therefore yields that
\begin{equation}\label{eq:gc_ge_line}
g_{c_0}(u) \ge B - p_0 u \qquad \text{for all } u \in \R.
\end{equation}
In particular, $g_{c_0}(u) \ge J_G(u)$ for $u \in [0,a]$.

For $u \ge a$, define then the difference
\[
D(u) := g_{c_0} (u)- \int_u^\infty s_G (t) \, \mathrm{d}t.
\]
By Lemma~\ref{lem:gaussian_call}, $D$ is differentiable and (noting that $u \ge a > t_0$, whereby $s_G(t) = 2 e^{-t^2/2}$) we can compute
\begin{align*}
D^{\prime} (u)
&= - \Phib (u/c_0) + 2e^{-u^2/2}
= 2e^{-u^2/2} \bigl(1 - R (u) \bigr),
\end{align*}
where $R (u) := \Phib(u/c_0)\big/(2e^{-u^2/2})$. 
By Lemma~\ref{lem:Rmono}, the function $R$ is increasing on
$[a, \infty)$. 
Moreover, by \eqref{eq:gc_ge_line} and the identity $\int_a^\infty s_G (t) \, \mathrm{d}t = B - a p_0$ (using that $H(a) = B$), we can check that
\[
D (a) = g_{c_0} (a) - \int_a^\infty s_G (t) \, \mathrm{d}t \ge (B - a p_0) - (B -a p_0) = 0.
\]
By inspection, one checks that $\lim_{u\to\infty} D (u) = 0$, and Lemma~\ref{lem:monotone_ratio_principle} therefore implies that $D (u) \ge 0$ for all $u\ge a$, i.e.\ $g_{c_0}(u) \ge J_G(u)$ as claimed.
\end{proof}

\begin{proof}[Proof of Theorem~\ref{thm:main}]
Let $Y := c_0 G$. Lemma~\ref{lem:gauss_ge_J} gives that $g_{c_0}(u) = \Ebb[(Y - u)_+]\ge J_G(u)$ for all $u \ge 0$.
Proposition~\ref{prop:envelope_to_cx} therefore yields $X \preceq_{cx} Y$.

For sharpness, let $X^\star$ be the extremizer from Lemma~\ref{lem:extremizer} for the sub-Gaussian envelope
$s_G (t) = \min \{1,2e^{-t^2/2}\}$, so that $\Pbb(|X^\star| > t) = s_G (t)$ for all $t \ge 0$ and $\Ebb [(X^\star - u)_+] = J_G (u)$ for all $u \ge 0$.
Fix $c\in (0, c_0)$ and set $u_c := c z$, where $z = \Phi^{-1} (1-p_0)$ as before.
Lemma~\ref{lem:J_ge_line} yields that $J_G (u_c) \ge B - p_0 u_c$. 
Using Lemma~\ref{lem:gaussian_call} and that $\Phib(z) = p_0$, compute that
\[
\Ebb [(X^\star - u_c)_+] - \Ebb[(c G - u_c)_+]
= J_G (u_c) - g_c (u_c)
\ge (B - p_0 u_c) - \bigl(c \, \phin(z) - p_0 u_c \bigr)
= B - c \, \phin (z)>0,
\]
because $c < c_0 = B / \phin(z)$. Taking $f(x) = (x-u_c)_+$ thus demonstrates that $X^\star \not \preceq_{cx} c G$, and it hence follows that the constant is unimprovable, i.e. $c_\star=c_0$.
\end{proof}

\section{Laplace domination under sub-exponential tail constraints}\label{sec:psi1}

Using the same tools, an analogous comparison is viable for random variables adhering to other tail constraints. In this section, we develop such a result under the assumption of a two-sided sub-exponential tail envelope %
\begin{equation}\label{eq:psi1_envelope}
s_E(t) : =\min \{1,2e^{-t}\},
\qquad t\ge 0
\end{equation}
for which the natural comparator is a scaled standard Laplace random variable $L$ with density $\frac12 e^{-|x|}$ on $\mathbb{R}$. 

In particular, by Proposition~\ref{prop:envelope_to_cx}, the sharp constant for the analog to Theorem~\ref{thm:main} can be determined by identifying the minimal $c_E > 0$ for which $g_{c_E L} (u) \ge J_E(u)$ for all $u \ge 0$, with $g$ the exact stop-loss for the Laplace random variable, and $J_E$ the stop-loss envelope for the family of random variables satisfying the sub-exponential tail constraint.

Towards establishing such a comparison, define
\[
A_E := \int_0^\infty s_E(t) \, \mathrm{d}t, \qquad B_E :=\frac{A_E}{2}, \qquad
H_E(x) := x \, s_E (x) + \int_x^\infty s_E (t) \, \mathrm{d}t.
\]
Let $a_E>0$ satisfy $H_E (a_E) = B_E$ and set $p_E: = s_E (a_E)$. 
Since $H_E(x) = 2e^{-x}(x+1)$ for $x > \log 2$ and $B_E = (\log 2 + 1)/2$, one checks that $H_E(2\log 2) = (2\log 2 + 1)/2 > B_E$, so $a_E > 2\log 2$ and $p_E = 2e^{-a_E} < 1$. 
Let $J_E$ denote the envelope $J_{s_E}$ from Lemma~\ref{lem:envelope}. 
Finally, define
\begin{equation}\label{eq:psi1_constants}
w_E := \log\frac{1}{2 p_E},
\qquad
c_E := \frac{B_E}{p_E ( 1 + w_E)}.
\end{equation}
Note that $w_E = a_E - 2 \log 2$ and $c_E > 1$. In particular, some extended (but elementary) calculations yield that $c_E \approx 1.89389433$.

\begin{theorem}[Sharp one-dimensional convex sub-exponential comparison]\label{thm:psi1}
Let $X$ be integrable with $\Ebb [X] = 0$ and assume the tail constraint
\[
\Pbb(|X| > t) \le \min\{1, 2e^{-t}\} \qquad \text{for all }t\ge 0.
\]
Then $X \preceq_{cx} c_E L$. Moreover, $c_E$ is optimal: for every $c < c_E$ there exists an integrable mean-zero
$X$ satisfying the same tail bound but with $X \not \preceq_{cx} cL$.
\end{theorem}

\begin{lemma}[Laplace stop-loss transform] \label{lem:laplace_call}
For $c > 0$, define $\ell_c(u) := \Ebb[(c L - u)_+]$. Then for all $u \ge 0$, one has the exact formulae
\[
\ell_c(u) = \frac{c}{2} e^{-u/c},
\qquad
\ell_c^{\prime} (u) = - \frac12 e^{-u/c} = -\Pbb(cL>u).
\]
In particular, $\ell_c$ is convex and decreasing on $[0, \infty)$.
\end{lemma}

\begin{proof}
For $u\ge 0$, compute that
\[
\ell_c(u)
= \int_{u/c}^\infty (c x - u) \, \frac12 e^{-x} \, \mathrm{d}x
= c\int_{u/c}^\infty (x - u/c) \, \frac12 e^{-x} \, \mathrm{d}x
=\frac{c}{2} e^{-u/c}.
\]
Differentiating gives the expression for $\ell_c^{\prime}(u)$.
\end{proof}

\begin{lemma}[Laplace stop-loss dominates the envelope] \label{lem:laplace_ge_J}
For all $u \ge 0$, we have $\ell_{c_E} (u) \ge J_E (u)$.
\end{lemma}

\begin{proof}
Set $u_\star :=c_E w_E$. 
By Lemma~\ref{lem:laplace_call}, $\ell_{c_E}$ is convex and differentiable on $[0, \infty)$ with $\ell_{c_E}^{\prime} (u) = - \frac12 e^{-u/c_E}$. 
Since $\Pbb(L > w_E) = \frac12 e^{-w_E} = p_E$, we have $\ell_{c_E}^{\prime}(u_\star) = - p_E$.
Moreover, by the definition of $c_E$, we see that
\[
\ell_{c_E} (u_\star) = \frac{c_E}{2} e^{-w_E} = c_E p_E = B_E - p_E u_\star.
\]
Lemma~\ref{lem:tangent} therefore yields that $\ell_{c_E}(u)\ge B_E - p_E u$ for all $u\ge 0$, and hence that
$\ell_{c_E} (u) \ge J_E (u)$ for $u \in [0, a_E]$.

For $u\ge a_E$, define the difference
\[
D(u) := \ell_{c_E} (u) - \int_u^\infty s_E (t) \, \mathrm{d}t.
\]
Since $a_E > \log 2$, we have $s_E (u) = 2e^{-u}$ for all $u \ge a_E$. Lemma~\ref{lem:laplace_call} then gives that
\[
D^{\prime} (u) = - \frac12 e^{-u/c_E} + 2 e^{-u} = 2 e^{-u} \bigl(1 - R (u)\bigr),
\qquad
R (u) := \frac14 \, e^{u (1 - 1/c_E)}.
\]
In particular, since $c_E > 1$, one sees that $R$ is increasing on $[a_E,\infty)$ (and even on all of $\mathbb{R}$). Also, since $H_E (a_E) = B_E$ and $s_E (a_E) = p_E$, we can check that
\[
\int_{a_E}^\infty s_E(t)\, \mathrm{d}t = B_E - a_E p_E
\qquad \text{and} \qquad 
J_E (a_E) = B_E - p_E a_E = B_E - a_E p_E.
\]
We thus see that $D (a_E) = \ell_{c_E} (a_E) - J_E (a_E)\ge 0$ by the first part of the proof. Moreover, by inspection, $\lim_{u\to\infty} D (u) = 0$, and so Lemma~\ref{lem:monotone_ratio_principle} therefore implies that $D(u) \ge 0$ for all $u \ge a_E$, i.e.\ $\ell_{c_E} (u) \ge J_E (u)$ as claimed.
\end{proof}

\begin{proof}[Proof of Theorem~\ref{thm:psi1}]
Let $Y: = c_E L$. 
Lemma~\ref{lem:laplace_ge_J} gives that $g_Y (u) = \Ebb [(Y - u)_+] = \ell_{c_E} (u)\ge J_E (u)$ for all $u \ge 0$.
Proposition~\ref{prop:envelope_to_cx} therefore yields that $X\preceq_{cx}Y$.

For sharpness, let $X^\star$ be the extremizer from Lemma~\ref{lem:extremizer} applied to the envelope $s_E$.
Then $\Pbb(|X^\star| > t) = s_E(t)$ for all $t\ge 0$ and $\Ebb[(X^\star - u)_+] = J_E (u)$ for all $u\ge 0$.
Fix $c\in (0, c_E)$ and set $u_c := c w_E$.
Lemma~\ref{lem:J_ge_line} then yields that $J_E (u_c) \ge B_E - p_E u_c$. 
Using Lemma~\ref{lem:laplace_call} and $\frac12 e^{-w_E} = p_E$, compute that
\[
\Ebb [(X^\star - u_c)_+]-\Ebb[(c L - u_c)_+]
= J_E (u_c) - \ell_c (u_c)
\ge (B_E - p_E u_c) - c \, p_E
= B_E - c p_E (1 + w_E) > 0,
\]
because $c < c_E = B_E / (p_E (1 + w_E))$. Taking $f(x) = (x - u_c)_+$ then shows that $X^\star \not \preceq_{cx} cL$, as claimed.
\end{proof}

\section{Higher-dimensional consequences}

Theorem~\ref{thm:main} is one-dimensional: our proof relies heavily on the stop-loss characterisation of the convex ordering, which does not immediately extend to higher dimension. Nevertheless, we record two applications to stylised high-dimensional settings. The first assumes a meaningful coordinate structure (a martingale decomposition); the second orders expectations only among a restricted class of convex functions.

\subsection*{Tensorization via a sequential (martingale) coordinate representation}

Write $[d]:=\{1,\dots,d\}$.  Throughout this subsection, $(\mathcal F_i)_{i=0}^d$ is a filtration.

\begin{theorem}[Sequential tensorization for convex domination]
\label{thm:sequential_tensorization}
Let $d \ge 1$.  Let $X_1, \dots, X_d$ be integrable real random variables such that $X_i$ is $\mathcal F_i$-measurable.
Let $Y_1, \dots, Y_d$ be integrable real random variables that are mutually independent and independent of $\mathcal F_d$.
Assume that for each $i \in [d]$ and every convex $\phi : \R \to \R$,
\begin{equation}
\label{eq:cond_cx_dom}
\Ebb\! \left[\phi(X_i) \, \middle| \, \mathcal F_{i-1} \right] \le \Ebb [ \phi (Y_i)]
\qquad\text{a.s.}
\end{equation}
Then, for every convex $f : \R^d \to \R$, it holds that
\begin{equation}
\label{eq:sequential_tensor_concl}
\Ebb \big[ f(X_1, \dots, X_d) \big] \le \Ebb \big[ f(Y_1, \dots, Y_d) \big],
\end{equation}
where both sides are understood as extended expectations in $(-\infty, \infty]$.
\end{theorem}

\begin{proof}
Fix a convex $f: \R^d \to \R$. Since $f$ is convex and proper, we are free to choose an affine minorant $\ell(x) = a^\top x + b \le f(x)$ and set $h := f - \ell \ge 0$.
Since $\ell$ is affine, invoking \eqref{eq:cond_cx_dom} with $\phi(t) = t$ and $\phi(t) = -t$ gives that
$\Ebb [X_i \mid \mathcal F_{i-1}] = \Ebb [Y_i]$ a.s., and hence that $\Ebb [\ell(X)] = \Ebb[ \ell (Y)]$.
It suffices to prove \eqref{eq:sequential_tensor_concl} with $f$ replaced by $h$, i.e. to restrict attention to non-negative convex test functions.

Define $h_d := h$ and for $i = 1, \dots, d$ set
\[
h_{i-1} (x_1, \dots, x_{i-1})
:= \Ebb \big[ h_i (x_1, \dots, x_{i-1}, Y_i)\big].
\]
Each $h_i$ is convex and nonnegative.
Fix $i \in [d]$. Since $(X_1, \dots, X_{i-1})$ is $\mathcal F_{i-1}$-measurable and $t \mapsto h_i (X_1, \dots, X_{i-1}, t)$ is convex, \eqref{eq:cond_cx_dom} yields that
\[
\Ebb\! \left[ h_i (X_1, \dots, X_{i-1}, X_i) \, \middle| \, \mathcal F_{i-1} \right]
\le \Ebb \big[ h_i (X_1, \dots, X_{i-1}, Y_i) \big]
= h_{i-1} (X_1, \dots, X_{i-1}),
\]
using independence of $Y_i$ from $\mathcal F_{i-1}$.
Taking expectations then gives that
$\Ebb [ h_i (X_1, \dots, X_i)] \le \Ebb[ h_{i-1} (X_1, \dots, X_{i-1})]$.
Iterating from $i = d$ down to $i = 1$ yields that $\Ebb [h (X)] \le \Ebb[h_0]$. 
By construction and independence of the $Y_i$, $\Ebb[h_0]=\Ebb[h(Y)]$, and so we conclude.
\end{proof}

\begin{lemma}[Conditional form of Theorem~\ref{thm:main}]
\label{lem:conditional_thm_main}
Let $\mathcal G$ be a sub-$\sigma$-field and let $Z$ be integrable.
Assume that $\Ebb[ Z \mid \mathcal G]=0$ a.s. and
\[
\Pbb \big( |Z| > t \mid \mathcal G \big) \le 2e^{-t^2/2}
\qquad \text{a.s. for all } t \ge 0.
\]
Let $G \sim \cN(0,1)$ be independent of $\mathcal G$.
Then for every convex $\phi : \R \to \R$, it holds that
\[
\Ebb\! \left[ \phi (Z) \, \middle| \, \mathcal G \right] \le \Ebb \big[ \phi (c_0 G) \big]
\qquad \text{a.s.}
\]
\end{lemma}

\begin{proof}
Fix convex $\phi$ and choose an affine minorant $\ell(t) = \alpha t + \beta \le \phi(t)$.
Set $\psi: = \phi - \ell \ge 0$.
Since $G$ is centered and $\Ebb[Z \mid \mathcal G] = 0$, we have that $\Ebb[ \ell (Z) \mid \mathcal G] = \beta$ a.s. and $\Ebb [ \ell (c_0G)] = \beta$.
Now, let $\nu_\omega$ be a regular conditional law of $Z$ given $\mathcal G$.
For $\Pbb$-a.e.\ $\omega$, the one-dimensional law $\nu_\omega$ satisfies the tail constraint \eqref{eq:subg_tail}, so
Theorem~\ref{thm:main} applied to $\nu_\omega$ yields that
$\int \psi\, \mathrm{d}\nu_\omega \le \Ebb [ \psi (c_0 G)]$.
Since $\psi \ge 0$, $\int \psi \, \mathrm{d}\nu_\omega = \Ebb[\psi(Z) \mid \mathcal G](\omega)$.
Adding back $\ell$ gives the claim.
\end{proof}

\begin{corollary}[Dimension-free domination from a martingale-coordinate representation]
\label{cor:martingale_basis}
Fix an orthonormal basis $u_1, \dots, u_d$ of $\R^d$ and let $X \in \R^d$ be integrable.
Define $\xi_i := \langle u_i, X \rangle$ and $\mathcal F_i := \sigma (\xi_1, \dots, \xi_i)$.
Assume that for each $i\in[d]$, there hold the conditional centering and  tail constraints
\[
\Ebb [ \xi_i \mid \mathcal F_{i-1}] = 0
\quad \text{and} \quad
\Pbb \big(|\xi_i| > t \mid \mathcal F_{i-1} \big) \le \min \{ 1,  2e^{-t^2/2} \}
\ \ \text{a.s. for all } t \ge 0.
\]
Let $G \sim \cN (0, I_d)$ be standard Gaussian in $\R^d$.
Then for every convex $f : \R^d \to \R$, there holds the ordering
\[
\Ebb [f (X)] \le \Ebb[f (c_0G)],
\]
where both sides are understood as extended expectations in $(-\infty, \infty]$.
\end{corollary}

\begin{proof}
Let $(G_1, \dots, G_d)$ be i.i.d.\ $\cN (0,1)$ independent of $\mathcal F_d$ and set $Y_i: = c_0 G_i$.
Lemma~\ref{lem:conditional_thm_main} applied with $\mathcal G = \mathcal F_{i-1}$ and $Z = \xi_i$ yields that
\[
\Ebb\! \left[ \phi (\xi_i) \, \middle| \, \mathcal F_{i-1} \right] \le \Ebb[\phi (Y_i)]
\qquad \text{a.s. for every convex } \phi : \R \to \R.
\]
Applying then Theorem~\ref{thm:sequential_tensorization} to (overloading notation momentarily) $(X_1, \dots, X_d) = (\xi_1, \dots, \xi_d)$ and $(Y_1, \dots, Y_d)$, we obtain that for every convex $\widetilde f : \R^d \to \R$,
\[
\Ebb[\widetilde f( \xi_1, \dots, \xi_d)] \le \Ebb[\widetilde f (Y_1, \dots, Y_d)].
\]
Take $\widetilde f (z_1, \dots, z_d) := f(\sum_{i=1}^d z_i u_i)$ to obtain that
$\Ebb [f (X)] \le \Ebb[f (\sum_{i=1}^d Y_i u_i)]$.
Since $\sum_{i=1}^d G_i u_i\stackrel{d}{=}G$, the right-hand side equals $\Ebb[f(c_0G)]$.
\end{proof}

We emphasise that this is a genuine multivariate convex ordering, meaning that the conclusion holds for all convex $f : \R^d \to \R$.

\subsection*{Domination for the cone generated by convex ridge functions}

\begin{theorem}[Convex domination for nonnegative ridge combinations]
\label{thm:ridge_cone_domination}
Let $X\in\R^d$ be integrable and satisfy the vector tail bound
\begin{equation}
\label{eq:vector_subg_tail}
\Ebb [X] = 0
\qquad\text{and}\qquad
\Pbb \big( | \langle v, X \rangle| > t \big) \le 2e^{-t^2/2}
\quad \text{for all } v \in \R^d \text{ with } \|v\|_2 = 1,\ t \ge 0.
\end{equation}
Let $G \sim \cN(0,I_d)$.
Fix a measurable $f : \R^d \to \R$ that admits a representation
\begin{equation}
\label{eq:ridge_rep}
f(x) = b + a^\top x + \sum_{k=1}^m \lambda_k \, \phi_k (\langle u_k, x \rangle),
\end{equation}
with $m \in \N$, $a, u_k \in \R^d$, $b \in \R$, $\lambda_k \ge 0$, and each $\phi_k :\R \to \R$ convex, or is a pointwise increasing limit of such functions.
Then there holds the ordering
\[
\Ebb [f (X)] \le \Ebb [f (c_0 G)],
\]
where both sides are understood as extended expectations in $(-\infty, \infty]$.
\end{theorem}

\begin{remark}[Ridge Convexity]
The cone of functions considered herein could be termed the class of `ridge-convex' functions. In dimensions $d \ge 2$, these form a strict subset of the cone of all convex functions on $\mathbb{R}^d$. As such, the `linear convex' ordering established by Theorem~\ref{thm:ridge_cone_domination} is in general strictly weaker than the `full' convex ordering which one might seek (and indeed, which the result of \cite{vanHandelSubgComparison} encourages one to seek).
\end{remark}

\begin{proof}
\emph{Step 1: the finite ridge case.}
Assume $f$ has the form \eqref{eq:ridge_rep}.
For each $k$, pick an affine minorant $\ell_k(t) = \alpha_k t + \beta_k \le \phi_k (t)$ and set $\psi_k := \phi_k-\ell_k \ge 0$.
Rewrite
\[
f(x) = b^{\prime}+(a^{\prime})^\top x + \sum_{k=1}^m \lambda_k \, \psi_k (\langle u_k, x \rangle),
\]
where $b^{\prime} := b + \sum_{k=1}^m \lambda_k\beta_k$ and $a^{\prime} := a + \sum_{k=1}^m \lambda_k \alpha_k u_k$, i.e. we again reduce to the setting of non-negative convex test functions.
Since $\psi_k\ge 0$ and $\lambda_k\ge 0$, Tonelli's Theorem gives that
\begin{equation}
\label{eq:expectation_tonelli}
\Ebb[f (Z)]
= b^{\prime} + (a^{\prime})^\top \Ebb[Z] + \sum_{k=1}^m \lambda_k \, \Ebb \big[\psi_k (\langle u_k, Z \rangle )\big]
\qquad \text{for } Z \in \{X, c_0 G\}.
\end{equation}
The affine term vanishes for both $Z = X$ and $Z = c_0G$ by \eqref{eq:vector_subg_tail}.

Fix $k$ with $u_k \ne 0$ and set $v_k := u_k / \|u_k\|_2$.
Then $Z_k := \langle v_k, X\rangle$ satisfies \eqref{eq:subg_tail} by \eqref{eq:vector_subg_tail}.
Apply Theorem~\ref{thm:main} to $Z_k$ with the convex test function $\theta \mapsto \psi_k( \|u_k\|_2\, \theta)$ to obtain that
\[
\Ebb \big[ \psi_k (\langle u_k, X \rangle) \big] \le \Ebb \big[ \psi_k(\langle u_k, c_0 G \rangle) \big].
\]
Summing over $k$ in \eqref{eq:expectation_tonelli} yields that $\Ebb [f (X)] \le \Ebb[f (c_0 G)]$ for ridge sums.

\smallskip
\emph{Step 2: monotone limits.}
Let $f_n$ be ridge sums with $f_n\uparrow f$ pointwise.
Choose an affine minorant $\ell \le f_1$ and set $g_n := f_n-\ell \uparrow g := f - \ell$, so $g_n, g \ge 0$.
Step~1 gives that $\Ebb [g_n (X)] \le \Ebb [g_n (c_0 G)]$, and monotone convergence hence yields that $\Ebb [g (X)] \le \Ebb [g (c_0 G)]$; adding back $\ell$ (whose expectations match) then gives the final claim.
\end{proof}

\bibliographystyle{alpha}
\bibliography{references}

\end{document}